\documentclass[a4paper,10pt,reqno]{amsart}

\usepackage{graphicx}
\usepackage{amsmath}
\usepackage{amsthm}
\usepackage{amssymb}
\usepackage{dsfont}
\usepackage{enumitem}
\usepackage[all,2cell]{xy}
\usepackage{hyperref}

\UseAllTwocells
\SilentMatrices
\setlist[enumerate,1]{leftmargin=*,label=\upshape{(\arabic*)}}
\numberwithin{equation}{section}

\theoremstyle{plain}
\newtheorem{thm}{Theorem}[section]
\newtheorem{prop}[thm]{Proposition}
\newtheorem{cor}[thm]{Corollary}
\newtheorem{lem}[thm]{Lemma}
\theoremstyle{definition}

\newtheorem{exm}[thm]{Example}
\theoremstyle{remark}

\DeclareMathOperator{\Hom}{Hom}
\DeclareMathOperator{\Ext}{Ext}

\DeclareMathOperator{\End}{End}

\DeclareMathOperator{\Ker}{Ker}
\DeclareMathOperator{\Coker}{Coker}
\renewcommand{\Im}{\operatorname{Im}}
\renewcommand{\top}{\operatorname{top}}

\DeclareMathOperator{\soc}{soc}
\DeclareMathOperator{\syz}{\Omega}
\DeclareMathOperator{\Tr}{Tr}
\DeclareMathOperator{\add}{add}

\renewcommand{\-}{\mbox{-}}
\renewcommand{\b}{\mathrm{b}}
\newcommand{\To}{\xrightarrow}

\newcommand{\sg}{\mathrm{sg}}
\newcommand{\op}{\mathrm{op}}

\newcommand{\K}{\mathbf{K}}
\newcommand{\D}{\mathbf{D}}
\newcommand{\perf}{\mathbf{perf}}
\newcommand{\thick}{\mathrm{thick}}

\renewcommand{\mod}{\mathrm{mod}}

\newcommand{\proj}{\mathrm{proj}}

\newcommand{\inj}{\mathrm{inj}}

\newcommand{\LA}{\mathbb{A}}

\newcommand{\Z}{\mathds{Z}}

\renewcommand{\AA}{\mathcal{A}}
\renewcommand{\SS}{\mathcal{S}}
\newcommand{\EE}{\mathcal{E}}
\newcommand{\FF}{\mathcal{F}}
\newcommand{\PP}{\mathcal{P}}
\newcommand{\WW}{\mathcal{W}}
\newcommand{\XX}{\mathcal{X}}

\begin{document}
\title{The singularity category of a Nakayama algebra}
\thanks{Supported by the National Natural Science Foundation of China (No. 11201446).}
\author{Dawei Shen}
\address{School of Mathematical Sciences \\
         University of Science and Technology of China \\
         Hefei, Anhui 230026 \\
         P. R. China}
\email{sdw12345@mail.ustc.edu.cn}
\urladdr{http://home.ustc.edu.cn/~sdw12345/}
\date{\today}
\begin{abstract}
  Let $A$ be a Nakayama algebra. We give a description of the singularity category of $A$ inside its stable module category. We prove that there is a duality between the singularity category of $A$ and the singularity category of its opposite algebra. As a consequence, the resolution quiver of $A$ and the resolution quiver of its opposite algebra have the same number of cycles and the same number of cyclic vertices.
\end{abstract}
\maketitle

\section{Introduction}
Let $A$ be an artin algebra. Denote by $A\-\mod$ the category of finitely generated left $A$-modules, and by $\D^\b(A\-\mod)$ the bounded derived category of $A\-\mod$. Recall that a complex in $\D^\b(A\-\mod)$ is \emph{perfect} provided that it is isomorphic to a bounded complex of finitely generated projective $A$-modules. Following \cite{Buc1987, Hap1991, Orl2004}, the \emph{singularity category} $\D_\sg(A)$ of $A$ is the quotient triangulated category of $\D^\b(A\-\mod)$ with respect to the full subcategory consisting of perfect complexes. Recently, the singularity category of a Nakayama algebra is described in \cite{CY2013}.

Let $A$ be a connected Nakayama algebra without simple projective modules. Following \cite{Rin2013}, the \emph{resolution quiver} $R(A)$ of $A$ is defined as follows: the vertex set is the set of isomorphism classes of simple $A$-modules, and there is an arrow from $S$ to $\tau \soc P(S)$ for each simple $A$-module $S$; see also \cite{Gus1985}. Here, $P(S)$ is the projective cover of $S$, `soc' denotes the socle of a module, and $\tau = D \Tr$ is the Auslander-Reiten translation \cite{ARS1995}. A simple $A$-module is called \emph{cyclic} provided that it lies in a cycle of $R(A)$.

The following consideration is inspired by \cite{Rin2013}. Let $A$ be a connected Nakayama algebra of infinite global dimension. Let $\SS_c$ be a complete set of pairwise non-isomorphic cyclic simple $A$-modules. Let $\XX_c$ be the set formed by indecomposable $A$-modules $X$ such that $\top X$ and $\tau\soc X$ both belong to $\SS_c$. Here, `top' denotes the top of a module. Denote by $\FF$ the full subcategory of $A\-\mod$ whose objects are finite direct sums of objects in $\XX_c$. It turns out that $\FF$ is a Frobenius abelian category, and it is equivalent to $A'\-\mod$ with $A'$ a connected selfinjective Nakayama algebra. Denote by $\underline{\FF}$ the stable category of $\FF$ modulo projective objects; it is a triangulated category by \cite{Hap1988}. We emphasize that the stable category $\underline{\FF}$ is a full subcategory of the stable module category $A\-\underline{\mod}$ of $A$.

The well-known result of \cite{Buc1987, Hap1991} describes the singularity category of a Gorenstein algebra $A$ via the subcategory of $A\-\underline\mod$ formed by Gorenstein projective modules. Here, we recall that an artin algebra is \emph{Gorenstein} if the injective dimension of the regular module is finite on both sides. In general, a Nakayama algebra is not Gorenstein \cite{Rin2013, CY2013}. The following result describes the singularity category of a Nakayama algebra via the subcategory $\underline{\FF}$ of $A\-\underline{\mod}$. For a Gorenstein Nakayama algebra, these two descriptions coincide; compare \cite{Rin2013}.

\begin{thm}
Let $A$ be a connected Nakayama algebra of infinite global dimension. Then
the singularity category $\D_\sg(A)$ and the stable category $\underline{\FF}$ are triangle equivalent.
\end{thm}

Denote by $A\-\inj$ the category of finitely generated injective $A$-modules, and by $\K^\b(A\-\inj)$ the bounded homotopy category of $A\-\inj$. We view $\K^\b(A\-\inj)$ as a thick subcategory of $\D^\b(A\-\mod)$ via the canonical functor. Then the quotient triangulated category $\D^\b(A\-\mod)/\K^\b(A\-\inj)$ is triangle equivalent to the opposite category of the singularity category $\D_\sg(A^\op)$ of $A^\op$. Here, $A^\op$ is the opposite algebra of $A$. In general, it seems that for an arbitrary artin algebra $A$, there is no obvious relation between $\D_\sg(A)$ and $\D_\sg(A^\op)$. However, we have the following result for a Nakayama algebra.

\begin{prop}
Let $A$ be a Nakayama algebra. Then the singularity category $\D_\sg(A)$ is triangle equivalent to $\D^\b(A\-\mod)/\K^\b(A\-\inj)$. Equivalently, there is a triangle duality between $\D_\sg(A)$ and $\D_\sg(A^\op)$.
\end{prop}

Let $A$ be a connected Nakayama algebra of infinite global dimension. Recall from \cite{Shen2014} that the resolution quivers $R(A)$ and $R(A^\op)$ have the same number of cyclic vertices. The following result strengthens the previous one by a different method.

\begin{prop}
Let $A$ be a connected Nakayama algebra of infinite global dimension. Then the resolution quivers $R(A)$ and $R(A^\op)$ have the same number of cycles and the same number of cyclic vertices.
\end{prop}

The paper is organized as follows. In Section 2, we recall some facts on singularity categories of artin algebras and the simplification in the sense of \cite{Rin1976}. In Section 3, we introduce the Frobenius subcategory $\FF$ and prove Theorem 1.1. The proofs of Propositions 1.2 and 1.3 are given in Sections 4 and 5, respectively.

Throughout this paper, we fix a commutative artinian ring $R$. All categories, morphisms and functors are $R$-linear.

\section{Preliminaries}
We first recall some facts on the singularity category of an artin algebra.

Let $A$ be an artin algebra over $R$. Recall that $A\-\mod$ denotes the category of finitely generated left $A$-modules. Let $A\-\proj$ denote the full subcategory consisting of projective $A$-modules, and $A\-\inj$ the full subcategory consisting of injective $A$-modules. Denote by $A\-\underline{\mod}$ the projectively stable category of finitely generated $A$-modules; it is obtained from $A\-\mod$ by factoring out the ideal of all maps which factor through projective $A$-modules; see \cite[IV.1]{ARS1995}.

Recall that for an $A$-module $M$, its \emph{syzygy} $\syz(M)$ is the kernel of its projective cover $P(M) \to M$. This gives rise to the \emph{syzygy functor} $\syz: A\-\underline{\mod} \to A\-\underline{\mod}$. Let $\syz^0(M) = M$ and $\syz^{i+1}(M) = \syz(\syz^i(M))$ for $i \geq 0$. Denote by $\syz^i(A\-\mod)$ the full subcategory of $A\-\mod$ formed by modules $M$ such that there is an exact sequence $0 \to M \to P_{i-1} \to \cdots \to P_{1} \to P_0$ with each $P_j$ projective. We also denote by $\syz^i_0(A\-\mod)$ the full subcategory of $\syz^i(A\-\mod)$ formed by modules without indecomposable projective direct summands.

Recall that $\D^\b(A\-\mod)$ denotes the bounded derived category of $A\-\mod$, whose translation functor is denoted by $[1]$. For each integer $n$, let $[n]$ denote the $n$-th power of $[1]$. The category $A\-\mod$ is viewed as a full subcategory of $\D^\b(A\-\mod)$ by identifying an $A$-module with the corresponding stalk complex concentrated at degree zero. Recall that a complex in $\D^\b(A\-\mod)$ is \emph{perfect} provided that it is isomorphic to a bounded complex of finitely generated projective $A$-modules. Perfect complexes form a thick subcategory of $\D^\b(A\-\mod)$, which is denoted by $\perf(A)$. Here, we recall that a triangulated subcategory is \emph{thick} if it is closed under direct summands.

Following \cite{Buc1987, Hap1991, Orl2004}, the quotient triangulated category
\[\D_\sg(A) = \D^\b(A\-\mod)/\perf(A)\]
is called the \emph{singularity category} of $A$. Denote by $q: \D^\b(A\-\mod) \to \D_\sg(A)$ the quotient functor. We recall that the objects in $\D_\sg(A)$ are bounded complexes of finitely generated $A$-modules. The translation functor of $\D_\sg(A)$ is also denoted by $[1]$.

The following results are well known.

\begin{lem}[{\cite[Lemma~2.1]{Chen2011b}}]\label{lem2.1}
Let $X$ be a complex in $\D_\sg(A)$ and $n_0>0$. Then for any $n$ sufficiently large, there exists a module $M$ in $\syz^{n_0}(A\-\mod)$ such that $X \simeq q(M)[n]$.
\end{lem}

\begin{lem}[{\cite[Lemma~2.2]{Chen2011b}}]\label{lem2.2}
Let $0 \to N \to P_{n-1} \to \cdots \to P_0 \to M \to 0$ be an exact sequence in $A\-\mod$ with each $P_i$ projective. Then there is an isomorphism $q(M) \simeq q(N)[n]$ in $\D_\sg(A)$. In particular, there is a natural isomorphism $\theta_M^n: q(M) \simeq q(\syz^n(M))[n]$ for any $M$ in $A\-\mod$ and $n \geq 0$.
\end{lem}

Observe that the composition $A\-\mod \to \D^\b(A\-\mod) \To{q} \D_\sg(A)$ vanishes on projective modules. Then it induces a unique functor $q': A\-\underline{\mod} \to \D_\sg(A)$. It follows from Lemma~\ref{lem2.2} that for each $n \geq 0$, the following diagram of functors
\[\xymatrix{
A\-\underline{\mod} \ar[d]_{q'} \ar[r]^{\syz^n} &A\-\underline{\mod}\ar[d]^{q'} \\
\D_\sg(A)\ar[r]^{[-n]} &\D_\sg(A)}\]
is commutative. Let $M$ and $N$ be in $A\-\mod$ and $n \geq 0$. Lemma~\ref{lem2.2} yields a natural map
\[\Phi^n: \underline{\Hom}_A(\syz^n(M), \syz^n(N)) \longrightarrow \Hom_{\D_\sg(A)}(q(M), q(N)).\]
Here, $\Phi^0$ is induced by $q'$ and $\Phi^n(f) = (\theta^n_N)^{-1}\circ(q'(f)[n])\circ\theta^n_M$ for $n \geq 1$.

Consider the following chain of maps $\{G^{n,n+1}\}_{n \geq 0}$ such that
\[G^{n,n+1}: \underline{\Hom}_A(\syz^n(M),\syz^n(N)) \longrightarrow \underline{\Hom}_A(\syz^{n+1}(M),\syz^{n+1}(N))\]
is induced by the syzygy functor $\syz$. The sequence $\{\Phi^n\}_{n \geq 0}$ is compatible with $\{G^{n,n+1}\}_{n\geq0}$, that is, $\Phi^{n+1} \circ G^{n,n+1} = \Phi^n$ for each $n \geq 0$. Then we obtain an induced map
\[\Phi: \varinjlim_{n\geq0}\underline{\Hom}_A(\syz^n(M),\syz^n(N)) \longrightarrow \Hom_{\D_\sg(A)}(q(M), q(N)).\]

\begin{lem}[{\cite[Exemple~2.3]{KV1987}}]\label{lem2.3}
Let $M$ and $N$ be in $A\-\mod$. Then there is a natural isomorphism
\[\Phi: \varinjlim_{n\geq0}\underline{\Hom}_A(\syz^n(M),\syz^n(N)) \overset{\simeq}\longrightarrow \Hom_{\D_\sg(A)}(q(M),q(N)).\]
\end{lem}

Next we recall the \emph{simplification} in the sense of \cite{Rin1976}.

Let $\AA$ be an abelian category. Recall that an object $X$ in $\AA$ is a \emph{brick} if $\End_{\AA}(X)$ is a division ring. Two objects $X$ and $Y$ are \emph{orthogonal} if $\Hom_{\AA}(X,Y) = 0$ and $\Hom_{\AA}(Y,X) = 0$. A full subcategory $\WW$ of $\AA$ is called a \emph{wide subcategory} if it is closed under kernels, cokernels and extensions. In particular, $\WW$ is an abelian category and the inclusion functor is exact. Recall that an abelian category $\AA$ is called a \emph{length category} provided that each object in $\AA$ has a composition series.

Let $\EE$ be a set of objects in an abelian category $\AA$. For an object $C$ in $\AA$, an \emph{$\EE$-filtration} of $C$ is given by a sequence of subobjects
\[0 = C_0 \subseteq C_1 \subseteq C_2 \subseteq \cdots \subseteq C_m = C,\]
such that each factor $C_i/C_{i-1}$ belongs to $\EE$ for $1 \leq i \leq m$. Denote by $\FF(\EE)$ the full subcategory of $\AA$ formed by objects in $\AA$ with an $\EE$-filtration.

\begin{lem}[{\cite[Theorem~1.2]{Rin1976}}]\label{lem2.4}
Let $\EE$ be a set of pairwise orthogonal bricks in $\AA$. Then $\FF(\EE)$ is a wide subcategory of $\AA$; moreover, $\FF(\EE)$ is a length category and $\EE$ is a complete set of pairwise non-isomorphic simple objects in $\FF(\EE)$.
\end{lem}

Let $A$ be a connected Nakayama algebra without simple projective modules. Recall that the vertex set of the \emph{resolution quiver} $R(A)$ of $A$ is the set of isomorphism classes of simple $A$-modules, and there is an arrow from $S$ to $\gamma(S) = \tau\soc P(S)$ for each simple $A$-module $S$. Since each vertex in $R(A)$ is the start of a unique arrow, each connected component of $R(A)$ contains precisely one cycle. A simple $A$-module is called \emph{cyclic} provided that it lies in a cycle of $R(A)$.

Let $A$ be a connected Nakayama algebra of infinite global dimension. In particular, $A$ has no simple projective modules. Let $\SS$ be a complete set of pairwise non-isomorphic simple $A$-modules. Denote by $\SS_c$ the subset of all cyclic simple $A$-modules, and by $\SS_{nc}$ the subset of all \emph{noncyclic} simple $A$-modules.

\begin{lem}\label{lem2.5}
Let $A$ be a connected Nakayama algebra of infinite global dimension. Then $\SS_c$ is a complete set of pairwise non-isomorphic simple $A$-modules of infinite injective dimension, and $\SS_{nc}$ is a complete set of pairwise non-isomorphic simple $A$-modules of finite injective dimension.
\end{lem}

\begin{proof}
This is dual to \cite[Corollaries~3.6 and 3.7]{Mad2005}.
\end{proof}

We will need the following fact. Recall that `top' denotes the top a module.

\begin{lem}[{\cite[Corollary to Lemma~2]{Rin2013}}]\label{lem2.6}
Let $A$ be a connected Nakayama algebra without simple projective modules. Assume that $M$ is an indecomposable $A$-module and $m \geq 0$. Then either $\syz^{2m}(M) = 0$ or else $\top \syz^{2m}(M) = \gamma^m(\top M)$.
\end{lem}

\section{A Frobenius subcategory}
In this section, we introduce a Frobenius subcategory in the module category of a Nakayama algebra, whose stable category is triangle equivalent to the singularity category of the given algebra.

Throughout this section, $A$ is a connected Nakayama algebra of infinite global dimension. Denote by $n(A)$ the number of the isomorphism classes of simple $A$-modules. Denote by $l(M)$ the composition length of an $A$-module $M$. Recall that $\SS$ denotes a complete set of pairwise non-isomorphic simple $A$-modules, $\SS_c$ the subset of all cyclic simple $A$-modules and $\SS_{nc}$ the subset of all noncyclic simple $A$-modules. Observe that the map $\gamma$ restricts to a permutation on $\SS_c$. Let $\XX_c$ be the set formed by indecomposable $A$-modules $X$ such that both $\top X$ and $\tau\soc X$ belong to $\SS_c$.

The proof of the following well-known lemma uses the structure of indecomposable modules over Nakayama algebras; see \cite[IV.3 and VI.2]{ARS1995}. Each indecomposable $A$-module $X$ is uniserial, and it is uniquely determined by its top and its composition length. Its composition factors from the top are $S, \tau S, \cdots, \tau^{l-1}S$, where $S = \top X$ and $l = l(X)$. In particular, the projective cover $P(X)$ of $X$ is indecomposable.

\begin{lem}\label{lem3.1}
Let $M$ be an indecomposable $A$-module which contains a nonzero projective submodule $P$. Then $M$ is projective.
\end{lem}

\begin{proof}
Suppose that, on the contrary, $M$ is nonprojective. Then there is a proper surjective map $\pi: P(M) \to M$, where $P(M)$ is the projective cover of $M$. We have a proper surjective map $\pi^{-1}(P) \to P$; it splits, since  $P$ is projective. This is impossible, since $P(M)$ is indecomposable, and thus its submodule $\pi^{-1}(P)$ is indecomposable.
\end{proof}

\begin{lem}\label{lem3.2}
Let $f: X \to Y$ be a morphism in $\XX_c$. Then $\Ker f$, $\Coker f$ and $\Im f$ belong to $\XX_c \cup \{0\}$.
\end{lem}

\begin{proof}
We may assume that $f$ is nonzero. Then we have $\top(\Im f) = \top X$ and $\tau\soc(\Im f) = \tau\soc Y$, both of which belong to $\SS_c$. Thus, $\Im f$ belongs to $\XX_c$.

If $f$ is not a monomorphism, then $\top(\Ker f) = \tau\soc(\Im f) = \tau\soc Y$ and $\tau\soc(\Ker f) = \tau\soc X$. Thus, $\Ker f$ belongs to $\XX_c$.

If $f$ is not an epimorphism, then $\top(\Coker f) = \top Y$ and $ \tau\soc(\Coker f) = \top(\Im f) = \top X$. Thus, $\Coker f$ belongs to $\XX_c$.
\end{proof}

\begin{lem}\label{lem3.3}
Let $X$ be an object in $\XX_c$. If $0 \subsetneq X'' \subsetneq X' \subsetneq X$ are subobjects of $X$ such that $X'/X''$ belongs to $\XX_c$, then $X''$ and $X/X'$ belong to $\XX_c$.
\end{lem}

\begin{proof}
Since both $\top X'' = \tau\soc X'/X''$ and $\tau\soc X'' = \tau\soc X$ belong to $\SS_c$, it follows from the definition that $X''$ belongs to $\XX_c$. Similarly, since both $\top X/X' = \top X$ and $\tau\soc X/X' = \top X'/X''$ belong to $\SS_c$, it follows from the definition that $X/X'$ belongs to $\XX_c$.
\end{proof}

Denote by $\PP_c$ a complete set of projective covers of modules in $\SS_c$. We claim that $\PP_c$ is a subset of $\XX_c$. Indeed, we have $\top P(S) = S$ and $\tau\soc P(S) = \gamma(S)$, both of which belong to $\SS_c$. It follows that $\XX_c$ is closed under projective covers.

For each $S$ in $\SS_c$, let $E(S)$ denote the indecomposable $A$-module of the least composition length among those objects $X$ in $\XX_c$ with $\top X = S$. Inspired by \cite[Section~4]{Rin2013}, we call $E(S)$ the \emph{elementary module} associated to $S$. Denote by $\EE_c$ the set of elementary modules.
Recall that $\FF(\EE_c)$ is the full subcategory of $A\-\mod$ formed by $A$-modules with an $\EE_c$-filtration.

The \emph{support} of an  $A$-module $M$ is the subset of $\SS$ consisting of those simple $A$-modules appearing as a composition factor of $M$. For a set $\XX$ of $A$-modules, we denote by $\add \XX$ the full subcategory of $A\-\mod$ whose objects are direct summands of finite direct sums of objects in $\XX$.

The following result is in spirit close to \cite[Proposition~2]{Rin2013}. In particular, we prove that each elementary module $E$ is a brick and thus $l(E)\leq n(A)$.

\begin{prop}\label{prop3.4}
Let $A$ be a connected Nakayama algebra of infinite global dimension. Then the following statements hold.
\begin{enumerate}[ref=\theprop(\arabic*)]
  \item The set $\EE_c$ of elementary modules is a set of pairwise orthogonal bricks, and thus $\FF(\EE_c)$ is a wide subcategory of $A\-\mod$.
  \item $\FF(\EE_c) = \add \XX_c$, which is closed under projective covers.
  \item\label{prop3.4(3)} Let $E$ and $E'$ be elementary modules. Then $E = E'$ if and only if their supports have nonempty intersection.
\end{enumerate}
\end{prop}

\begin{proof}
(1) Let $f: E \to E'$ be a nonzero map between elementary modules. By Lemma~\ref{lem3.2} $\Im f$ belongs to $\XX_c$. However, $\Im f$ is a factor module of $E$. By the definition of the elementary module $E$ we have $E = \Im f$. Then $f$ is an injective map. Similarly, $f$ is a surjective map and thus an isomorphism. Therefore $\EE_c$ is a set of pairwise orthogonal bricks.

By Lemma~\ref{lem2.4} $\FF(\EE_c)$ is a wide subcategory of $A\-\mod$. In particular, it is closed under direct sums and direct summands.

(2) We prove that any module $X$ in $\XX_c$ belongs to $\FF(\EE_c)$, and thus $\add \XX_c \subseteq \FF(\EE_c)$. We use induction on $l(X)$. Set $S = \top X \in \SS_c$. If $X = E(S) \in \EE_c$, we are done. Otherwise, there is a proper surjective map $\pi: X \to E(S)$. By Lemma~\ref{lem3.2} we have $\Ker \pi \in \XX_c$. Then by induction $\Ker \pi \in \FF(\EE_c)$. Therefore $X \in \FF(\EE_c)$.

Recall that each elementary module $E$ satisfies that $\top E \in \SS_c$ and $\tau\soc E \in \SS_c$. It follows from its $\EE_c$-filtration that each indecomposable object $X$ in $\FF(\EE_c)$ satisfies that $\top X \in \SS_c$ and $\tau\soc X \in \SS_c$. Then by definition $X$ belongs to $\XX_c$. Therefore $\add \XX_c \supseteq \FF(\EE_c)$, and thus $\FF(\EE_c) = \add \XX_c$. Since $\XX_c$ is closed under projective covers, we infer that $\add \XX_c$ is closed under projective covers.

(3) Suppose that $E \neq E'$ have a common composition factor. We may assume that $l(E) \leq l(E')$. Since $E$ and $E'$ are orthogonal, there exists a chain $0 \subsetneq E_1 \subsetneq E_2 \subsetneq E'$ of $A$-modules such that $E_2/E_1 = E$. By Lemma~\ref{lem3.3} we have that $E'/E_2$ belongs to $\XX_c$. This contradicts to the definition of the elementary module $E'$.
\end{proof}

\begin{lem}\label{lem3.5}
Let $S$ be a cyclic simple $A$-module. Then the following statements hold.
\begin{enumerate}[ref=\thelem(\arabic*)]
  \item\label{lem3.5(1)} The injective dimension of $E(S)$ is infinite, and the injective dimension of $P(S)$ is finite.
  \item\label{lem3.5(2)} There is a unique simple $A$-module $S'$ in $\SS_c$ such that $\top E(S') = \tau\soc E(S)$ and $\Ext_A^1(E(S),E(S')) \neq 0$.
\end{enumerate}
\end{lem}

\begin{proof}
(1) We recall from Lemma~\ref{lem2.5} that $\SS_c$ is a complete set of pairwise non-isomorphic simple $A$-modules of infinite injective dimension. Since the elementary modules have pairwise disjoint supports, for each $S$ in $\SS_c$, the support of $E(S)$ contains precisely one cyclic simple $A$-module, that is, $S$. In other words, each composition factor of $E(S)$ different from $S$ is a noncyclic simple $A$-module, and thus has finite injective dimension. It follows that $E(S)$ has infinite injective dimension.

Let $h: P(S) \to I$ be an injective envelope of the $A$-module $P(S)$. We claim that each composition factor $S'$ of $\Coker h$ is a noncyclic simple $A$-module, and thus has finite injective dimension. Consequently, the injective dimension of $\Coker h$ is finite. Therefore the injective dimension of $P(S)$ is finite.

For the claim, we observe by Lemma~\ref{lem3.1} $P(S) \subsetneq P(S') \subseteq I$. Then we have $\gamma(S') = \gamma(S)$. Recall that the restriction of $\gamma$ on cyclic simple $A$-modules is injective. Therefore $S'$ is a noncyclic simple $A$-module, since $S$ is a cyclic simple $A$-module and $S' \neq S$.

(2) Let $E = E(S)$. Recall that $P(S)$ lies in $\XX_c$ and thus in $\FF(\EE_c)$. Consider the $\EE_c$-filtration of $P(S)$, say
\[0 = M_0 \subsetneq M_1 \subsetneq  \cdots \subsetneq  M_{t-1} \subsetneq M_t = P(S),\]
such that $M_i/M_{i-1}$ is elementary for $1 \leq i \leq t$. We observe that $M_t/M_{t-1} = E$ and $t \geq 2$, since by (1) we have $E(S)\neq P(S)$. Set $E' = M_{t-1}/M_{t-2}$. Note that $E' = E(S')$ for some cyclic simple $A$-module $S'$. Then
\[\top E' = \top(M_{t-1}/M_{t-2}) = \tau\soc(M_t/M_{t-1}) = \tau\soc E.\]
Since $M_t/M_{t-2} = P(S)/M_{t-2}$ is indecomposable, the exact sequence
\[0 \to M_{t-1}/M_{t-2} \to M_t/M_{t-2} \to M_t/M_{t-1} \to 0\]
does not split. Then we have $\Ext_A^1(E, E') \neq 0$. The uniqueness of $S'$ is obvious, since $S' = \tau\soc E(S)$.
\end{proof}

Recall that by definition $\tau\soc E$ lies in $\SS_c$ for each elementary module $E$. We have a map $\delta: \SS_c \to \SS_c$, which sends a cyclic simple $A$-module $S$ to $\delta(S) = \tau\soc E(S)$. We claim that $\delta$ is injective and thus bijective. Indeed, if $\delta(S) = \delta(\bar{S})$, then $\soc E(S) = \soc E(\bar{S})$. It follows from Lemma~\ref{prop3.4(3)} that $S = \bar{S}$.

\begin{cor}\label{cor3.6}
Let $S$ be a cyclic simple $A$-module and $t$ the minimal positive integer such that $\delta^t(S) = S$. Then $\SS_c = \{S, \delta(S), \cdots, \delta^{t-1}(S)\}$ and $\SS$ is the disjoint union of the supports of all elementary modules.
\end{cor}

\begin{proof}
Since $A$ is a connected Nakayama algebra without simple projective modules, any nonempty subset of $\SS$ which is closed under $\tau$ must be $\SS$. By Lemma~\ref{lem3.5(2)} the union of the supports of all $E(\delta^i(S))$ is closed under $\tau$, we infer that this union is $\SS$. Let $S'$ be a cyclic simple $A$-module. Then there exists an integer $0 \leq i \leq t-1$ such that the support of $E(\delta^i(S))$ and the support of $E(S)$ have nonempty intersection. It follows from Lemma~\ref{prop3.4(3)} that $S' = \delta^i(S)$.
\end{proof}

\begin{prop}\label{prop3.7}
Let $A$ be a connected Nakayama algebra of infinite global dimension. Then  $\FF(\EE_c)$ is equivalent to $A'\-\mod$, where $A'$ is a connected selfinjective Nakayama algebra.
\end{prop}

\begin{proof}
Let $P = \oplus_{S\in\SS_c}P(S)$ and $A' = \End_A(P)^\op$. Then $P$ is a projective object in $\FF(\EE_c)$, since $\FF(\EE_c)$ is a wide subcategory of $A\-\mod$. The natural projection $P(S) \to E(S)$ is a projective cover in the category $\FF(\EE_c)$. Recall from Lemma~\ref{lem2.4} that $\FF(\EE_c)$ is a length category with $\EE_c = \{E(S)\mid S \in \SS_c\}$ a complete set of pairwise non-isomorphic simple objects. We infer that for each object $X$ in $\FF(\EE_c)$, there is an epimorphism $P' \to X$ with $P'$ in $\add P$. Then $P$ is a projective generator for $\FF(\EE_c)$. We have an equivalence $\FF(\EE_c) \simeq A'\-\mod$; compare \cite[Chapter IV, Theorem~5.3]{Mit1965}.

Since each indecomposable object in $\FF(\EE_c)$ is uniserial, we infer that $A'$ is a Nakayama algebra. Denote by $\tau'$ the Auslander-Reiten translation of $A'$. Then we have $\tau'E(S) = E(\delta(S))$ by Lemma~\ref{lem3.5(2)}. It follows from Corollary~\ref{cor3.6} that all simple $A'$-modules are in the same $\tau'$-orbit. Therefore, the Nakayama algebra $A'$ is connected.

It remains to show that $A'$ is selfinjective. Since $\gamma$ restricts to a permutation on $\SS_c$, the modules in $\PP_c$ have pairwise distinct socles. Therefore, we have $\SS_c = \{\tau\soc P \mid P \in \PP_c\}$. Let $E$ be an elementary module. Since $\tau\soc E$ lies in $\SS_c$, there exists $P$ in $\PP_c$ with $\soc P = \soc E$ and $\soc P \neq \soc E'$ for any elementary module $E' \neq E$. It follows that the socle of $P$ in the category $\FF(\EE_c)$ is $E$. We have proven that every simple $A'$-module can embed into a projective $A'$-module. Therefore, $A'$ is selfinjective.
\end{proof}

The following result is analogous to \cite[Proposition~4]{Rin2013}.

\begin{lem}\label{lem3.8}
The following statements are equivalent for an indecomposable nonprojective $A$-module $M$.
\begin{enumerate}[ref=\thelem(\arabic*)]
  \item $M$ belongs to $\FF(\EE_c)$.
  \item\label{lem3.8(2)} There is an exact sequence $0 \to M \to P_n \to \cdots \to P_1 \to P_0 \to M \to 0$ for some $n \geq 1$ such that each $P_i$ belongs to $\PP_c$.
  \item There is an exact sequence $P_1 \to P_0 \to M \to 0$ such that $P_i$ belongs to $\PP_c$ for $i=0,1$.
\end{enumerate}
\end{lem}

\begin{proof}
``(1) $\Rightarrow$ (2)" Recall that for an indecomposable nonprojective module $M'$ over a selfinjective Nakayama algebra $A'$, there exists an exact sequence $0 \to M' \to P'_n \to \cdots \to P'_1 \to P'_0 \to M' \to 0$ for some $n \geq 1$ such that each $P'_i$ is indecomposable projective. Then (2) follows from Proposition~\ref{prop3.7}.

``(2) $\Rightarrow$ (3)" This is obvious.

``(3) $\Rightarrow$ (1)" Observe that $\top M = \top P_0$ and $\tau\soc M = \top \syz M = \top P_1$, both of which belong to $\SS_c$. Then by definition $M$ belongs to $\XX_c$.
\end{proof}

Recall that each component of the resolution quiver $R(A)$ has a unique cycle. For each noncyclic vertex $S$ in $R(A)$, there exists a unique path of minimal length starting with $S$ and ending in a cycle. We call the length of this path the distance between $S$ and the cycle. Let $d(A)$ be the maximal distance between noncyclic vertices and cycles. Observe that $\gamma^d(S)$ is cyclic for each simple $A$-module $S$.

\begin{lem}\label{lem3.9}
Let $d = d(A)$ be as above. Then the following statements hold.
\begin{enumerate}[ref=\thelem(\arabic*)]
  \item\label{lem3.9(1)} $\syz^{2d}(M)$ belongs to $\FF(\EE_c)$ for any $M$ in $A\-\mod$.
  \item $\syz_0^{2d}(A\-\mod) \subseteq \FF(\EE_c) \subseteq \syz^{2d}(A\-\mod)$.
\end{enumerate}
\end{lem}

\begin{proof}
(1) We may assume that $M$ is indecomposable. It follows from Lemma~\ref{lem2.6} that either $\syz^{2d}(M)$ is zero or $\top \syz^{2d}(M) = \gamma^d(\top M)$. If $\syz^{2d}(M)$ is indecomposable projective, then $\syz^{2d}(M)$ belongs to $\PP_c$. If $\syz^{2d}(M)$ is indecomposable nonprojective, then $\top \syz^{2d}(M) = \gamma^d(\top M)$ and $\tau\soc \syz^{2d}(M) = \top \syz^{2d+1}(M) = \gamma^d(\top \syz M)$. Then by definition $M$ belongs to $\XX_c$.

(2) The first inclusion follows from (1), and the second one follows from Lemma~\ref{lem3.8(2)}.
\end{proof}

By Proposition~\ref{prop3.7} $\FF(\EE_c)$ is a Frobenius category whose projective objects are precisely $\add \PP_c$. Denote by $\underline{\FF}(\EE_c)$ the stable category of $\FF(\EE_c)$ modulo projective objects. It is a triangulated category; see \cite{Hap1988}.

Recall from Proposition~\ref{prop3.4} that $\FF(\EE_c)$ is a wide subcategory of $A\-\mod$ which is closed under projective covers. Consider the inclusion functor $i: \FF(\EE_c) \to A\-\mod$. It induces uniquely a fully-faithful functor $i': \underline{\FF}(\EE_c) \to A\-\underline\mod$. We recall the induced functor $q': A\-\underline{\mod} \to \D_\sg(A)$ in Section~2.

The following is the main result of this section, which describes the singularity category of $A$ as a subcategory of the stable module category of $A$.

\begin{thm}\label{thm3.10}
Let $A$ be a connected Nakayama algebra of infinite global dimension. Then the composite functor $q'\circ{i}': \underline{\FF}(\EE_c) \to A\-\underline{\mod} \to \D_\sg(A)$ is a triangle equivalence.
\end{thm}

\begin{proof}
Observe that the composite functor $\FF(\EE_c) \To{i} A\-\mod \to \D^\b(A\-\mod) \To{q} \D_\sg(A)$ is a $\partial$-functor in the sense of \cite[Section~1]{Kel1991}; compare \cite[Lemma~2.4]{Chen2011a}. Then the functor $q' \circ i'$ is a triangle functor; see \cite[Lemma~2.5]{Chen2011a}.

Recall that the subcategory $\FF(\EE_c)$ of $A\-\mod$ is wide and closed under projective covers; moreover, $\FF(\EE_c)$ is a Frobenius category. Then the restriction of the syzygy functor $\syz: A\-\underline{\mod} \to A\-\underline{\mod}$ on $\underline{\FF}(\EE_c)$ is an autoequivalence, in particular, it is fully faithful. Then the functor $q'\circ i'$ is fully faithful by the natural isomorphism in Lemma~\ref{lem2.3}.

It remains to show that the functor $q'\circ i'$ is also dense. Let $X$ be an object in $\D_\sg(A)$. It follows from Lemmas \ref{lem2.1} and \ref{lem3.9(1)} that there exists a module $M$ in $\FF(\EE_c)$ and $n$ sufficiently large such that $X \simeq q(M)[n]$ in $\D_\sg(A)$. By above, the image $\Im(q'\circ i')$ is a triangulated subcategory of $\D_\sg(A)$, in particular, it is closed under $[m]$ for all $m \in \Z$. It follows from $X \simeq q(M)[n]$ that $X$ lies in $\Im(q'\circ i')$. This finishes our proof.
\end{proof}

We observe the following immediate consequence of Proposition~\ref{prop3.7} and Theorem~\ref{thm3.10}.

\begin{cor}[compare {\cite[Corollary~3.11]{CY2013}}]\label{cor3.11}
Let $A$ be a connected Nakayama algebra of infinite global dimension. Then there is a triangle equivalence between $\D_\sg(A)$ and $A'\-\underline{\mod}$ for a connected selfinjective Nakayama algebra $A'$.
\end{cor}

\section{A duality between singularity categories}
In this section, we prove that there is a triangle duality between the singularity category of a Nakayama algebra and the singularity category of its opposite algebra. The proof uses the Frobenius subcategory in the previous section.

Let $A$ be a connected Nakayama algebra of infinite global dimension. We recall from Propositions \ref{prop3.4} and \ref{prop3.7} that the category $\FF = \FF(\EE_c)$ is a wide subcategory of $A\-\mod$ closed under projective covers; it is equivalent to $A'\-\mod$ for a connected selfinjective Nakayama algebra $A'$.

Consider the inclusion functor $i: \FF \to A\-\mod$. We claim that it admits an exact right adjoint $i_\rho: A\-\mod \to \FF$.

For the claim, recall from the proof of Proposition~\ref{prop3.7} that $A' = \End_A(P)^\op$ with $P = \oplus_{S\in\SS_c}P(S)$. We identify $\FF$ with $A'\-\mod$. Then the inclusion $i$ is identified with $P\otimes_{A'}-$. The right adjoint is given by $i_\rho = \Hom_A(P,-)$. It is exact since ${}_AP$ is projective.

The adjoint pair $(i, i_\rho)$ induces an adjoint pair $(i^*, i_\rho^*)$ of triangle functors between bounded derived categories. Here,  for an exact functor $F$ between abelian categories, $F^*$ denotes its extension on bounded derived categories.

Recall that $\K^\b(A\-\inj)$ denotes the the bounded homotopy category of $A\-\inj$. We view $\K^\b(A\-\inj)$ as a thick subcategory of $\D^\b(A\-\mod)$ via the canonical functor. We mention that by the usual duality on module categories, the quotient triangulated category $\D^\b(A\-\mod)/\K^\b(A\-\inj)$ is triangle equivalent to the opposite category of the singularity category $\D_\sg(A^\op)$ of $A^\op$. Here, $A^\op$ is the opposite algebra of $A$.

The proof of the following result is similar to \cite[Propsition 2.13]{CY2013}. Recall that $\PP_c$ is a complete set of pairwise non-isomorphic indecomposable projective objects in $\FF = \FF(\EE_c)$.

\begin{lem}\label{lem4.1}
Let $A$ be a connected Nakayama algebra of infinite global dimension. Then the above functors $i_\rho^*$ and $i^*$ induce mutually inverse triangle equivalences between $\D^\b(A\-\mod)/\K^\b(A\-\inj)$ and $\D^\b(\FF)/\K^\b(\add \PP_c)$.
\end{lem}

\begin{proof}
Observe by \cite[Lemma 3.3.1]{CK2011} that $i^*: \D^\b(\FF) \to \D^\b(A\-\mod)$ is fully faithful. It follows that its right adjoint $i_\rho^*$ induces a triangle equivalence
\[\overline{i_\rho^*}: \D^\b(A\-\mod)/\Ker i_\rho^* \simeq \D^\b(\FF);\]
see \cite[Chapter I, Section 1, 1.3 Proposition]{GZ1967}. Here, $\Ker F$ denotes the essential kernel of an additive functor $F$.

We claim that $\Ker i_\rho^* = \thick\langle\SS_{nc}\rangle$, the smallest thick subcategory of $\D^\b(A\-\mod)$ containing $\SS_{nc}$. Here, we recall that $\SS_{nc}$ denotes the set of noncyclic simple $A$-modules. By Lemma~\ref{lem2.5} each noncyclic simple $A$-module has finite injective dimension. It follows from the claim that $\Ker i_\rho^* \subseteq \K^b(A\-\inj)$.

For the claim, we observe that $\Ker i_\rho = \FF(\SS_{nc})$, the full subcategory of $A\-\mod$ formed by $A$-modules with an $\SS_{nc}$-filtration. The claim follows from the fact that a complex $X$ is in $\Ker i_\rho^*$ if and only if each cohomology $H^i(X)$ is in $\Ker i_\rho$.

We observe that $i_\rho$ preserves injective objects since it has an exact left adjoint. It follows that $i_\rho^*(\K^\b(A\-\inj)) \subseteq \K^\b(\add \PP_c)$. By Lemma~\ref{lem3.5(1)} each module $Q$ in $\PP_c$ has finite injective dimension. Note that $i_\rho^*Q = Q$. Therefore $i_\rho^*(\K^\b(A\-\inj)) \supseteq \K^\b(\add \PP_c)$, and thus $i_\rho^*(\K^\b(A\-\inj)) = \K^\b(\add \PP_c)$. From this equality, the triangle equivalence $\overline{i_\rho^*}$ restricts to a triangle equivalence
\[\K^\b(A\-\inj)/\Ker i_\rho^* \simeq \K^\b(\add \PP_c).\]
The desired equivalence follows from \cite[Chapitre I, \S2, 4-3 Corollaire]{Ver1977}.
\end{proof}

\begin{prop}\label{prop4.2}
Let $A$ be a Nakayama algebra. Then the singularity category $\D_\sg(A)$ is triangle equivalent to $\D^\b(A\-\mod)/\K^\b(A\-\inj)$. Equivalently, there is a triangle duality between $\D_\sg(A)$ and $\D_\sg(A^\op)$.
\end{prop}

\begin{proof}
Without loss of generality, we may assume that $A$ is a connected Nakayama algebra of infinite global dimension. Then the singularity category $\D_\sg(A)$ is triangle equivalent to the stable category $\underline{\FF}$ by Theorem~\ref{thm3.10} .

Since $\FF$ is a Frobenius abelian category, it follows from \cite[Theorem~2.1]{Ric1989} that the stable category $\underline{\FF}$ is triangle equivalent to $\D^\b(\FF)/\K^\b(\add \PP_c)$. Then the conclusion follows from Lemma~\ref{lem4.1}.
\end{proof}

\section{The resolution quivers}
Let $A$ be a connected Nakayama algebra of infinite global dimension. Recall that $n(A)$ denotes the number of isomorphism classes of simple $A$-modules. By Corollary~\ref{cor3.11} the Auslander-Reiten quiver of the singularity category $\D_\sg(A)$ is isomorphic to a truncated tube $\Z\LA_m/\langle\tau^t\rangle$, where $m = m(A)$ denotes its height and $t = t(A)$ denotes its rank. Here, we use the fact that the Auslander-Reiten quiver of the stable module category of a connected selfinjective Nakayama algebra is a truncated tube; compare~\cite[VI.2]{ARS1995}.

Recall that $R(A)$ denotes the resolution quiver of $A$. We denote by $c(A)$ the number of cycles in $R(A)$. Let $C$ be a cycle in $R(A)$. Then the \emph{size} $s(C)$ of $C$ is the number of vertices in $C$, and the \emph{weight} $w(C)$ of $C$ is $\frac{\sum_Sl(P(S))}{n(A)}$, where $S$ runs though all vertices of $C$. Here, $l(P(S))$ is the composition length the projective cover $P(S)$ of a simple $A$-module $S$. Recall from~\cite{Shen2014} that all cycles in the resolution quiver $R(A)$ have the same size and the same weight. We denote $s(A) = s(C)$ and $w(A) = w(C)$ for an arbitrary cycle $C$ in $R(A)$.

For two positive integers $a$ and $b$, we denote their greatest common divisor by $(a,b)$.

\begin{lem}\label{lem5.1}
Let $m = m(A)$ and $t = t(A)$ be as above. Then $c(A) = (m+1,t)$, $s(A) = \frac{t}{(m+1,t)}$ and $w(A) = \frac{m+1}{(m+1,t)}$.
\end{lem}

\begin{proof}
Recall from \cite[Theorem~3.8]{CY2013} that there exists a sequence of algebra homomorphisms
\[A = A_0 \To{\eta_0} A_1 \To{\eta_1} A_2 \To{} \cdots \To{} A_{r-1} \To{\eta_{r-1}} A_r\]
such that each $A_i$ is a connected Nakayama algebra and $A_r$ is selfinjective; moreover, each $\eta_i$ induces a triangle equivalence between $\D_\sg(A_i)$ and $\D_\sg(A_{i+1})$. Following \cite[Lemma~2.2]{Shen2014}, each $\eta_i$ induces a bijection between the set of cycles in $R(A_i)$ and the set of cycles in $R(A_{i+1})$, which preserves sizes and weights. Then we have $m(A_i) = m(A_{i+1})$, $t(A_i) = t(A_{i+1})$, $c(A_i) = c(A_{i+1})$, $s(A_i) = s(A_{i+1})$ and $w(A_i) = w(A_{i+1})$ for $0 \leq i \leq r-1$. Therefore it is enough to prove the equations for selfinjective Nakayama algebras.

Let $A$ be a connected selfinjective Nakayama algebra. Then $t$ equals the number of isomorphism classes of simple $A$-modules, and $m+1$ equals the radical length of $A$. We claim that $s(A) = \frac{t}{(m+1,t)}$. Therefore, we have $c(A) = \frac{t}{s(A)} = (m+1,t)$ and $w(A) = \frac{(m+1)s(A)}{t} = \frac{m+1}{(m+1,t)}$.

For the claim, let $\{S_1, \cdots, S_t\}$ be a complete set of pairwise non-isomorphic simple $A$-modules such that $\tau S_i = S_{i+1}$ for $1 \leq i \leq t$. Here, we let $S_{t+j} = S_j$ for each $j > 0$. Then we have $\gamma(S_i) = S_{i+m+1}$, and thus $\gamma^d(S_i) = S_i$ if and only if $t$ divides $d(m+1)$. It follows that $R(A)$ consists of cycles of size $\frac{t}{(m+1,t)}$.
\end{proof}

The following result establishes the relationship between the resolution quiver of a Nakayama algebra and the resolution quiver of its opposite algebra.

\begin{prop}\label{prop5.2}
Let $A$ be a connected Nakayama algebra of infinite global dimension. Then the following statements hold.
\begin{enumerate}
  \item The resolution quivers $R(A)$ and $R(A^\op)$ have the same number of cycles and the same number of cyclic vertices.
  \item All cycles in $R(A)$ and $R(A^\op)$ have the same weight.
\end{enumerate}
\end{prop}

\begin{proof}
By Proposition~\ref{prop4.2} there is a triangle duality between $\D_\sg(A)$ and $\D_\sg(A^\op)$. Then the Auslander-Reiten quiver of $\D_\sg(A^\op)$ is isomorphic to the opposite quiver of the Auslander-Reiten quiver of $\D_\sg(A)$. Therefore, we have $m(A) = m(A^\op)$ and $t(A) = t(A^\op)$. It follows from Lemma~\ref{lem5.1} that $c(A) = c(A^\op)$, $s(A) = s(A^\op)$ and $w(A) = w(A^\op)$.
\end{proof}

The following example shows that these two resolution quivers $R(A)$ and $R(A^\op)$ may not be isomorphic in general.

\begin{exm}
Let $A$ be a connected Nakayama algebra with admissible sequence $(7,6,6,5)$. Assume that $\{S_1, S_2, S_3, S_4\}$ is a complete set of pairwise non-isomorphic simple $A$-modules such that $\tau S_i = S_{i+1}$ for $1 \leq i \leq 4$. Then we have $l(P_1) = 7$, $l(P_2) = 6$, $l(P_3) = 6$ and $l(P_4) = 5$. There is an arrow from $S_i$ to $S_j$ in $R(A)$ if and only if $4$ divides $i-j+l(P_i)$.

Denote by $D$ the usual duality, and by $(-)^*$ the duality on projectives. Then $\{DS_4, DS_3,DS_2,DS_1\}$ is a complete set of pairwise non-isomorphic simple $A^\op$-modules such that $\tau' DS_i = DS_{i-1}$. Here, $\tau'$ is the Auslander-Reiten translation of $A^\op$. We observe that $l(P_4^*) = 6$, $l(P_3^*) = 7$, $l(P_2^*) = 6$ and $l(P_1^*) = 5$. Therefore the admissible sequence of $A^\op$ is $(6,7,6,5)$. There is an arrow from $DS_i$ to $DS_j$ in $R(A^\op)$ if and only if $4$ divides $i-j-l(P_i^*)$.

The resolution quivers $R(A)$ and $R(A^\op)$ are shown as follows.
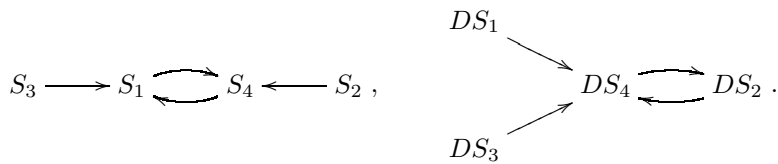
\[\begin{aligned}[c]
\xymatrix@R=10pt{
S_3 \ar@{->}[r] &S_1 \ar@/^/[r] &S_4 \ar@/^/[l]\ar@{<-}[r] &S_2}
\end{aligned}
,\qquad
\begin{aligned}[c]
\xymatrix@R=10pt{
DS_1 \ar@{->}[dr] \\
&DS_4 \ar@/^/[r] &DS_2 \ar@/^/[l] \\
DS_3 \ar@{->}[ur]}
\end{aligned}.\]
\end{exm}

\section*{Acknowledgements}
The author is grateful to his supervisor Professor Xiao-Wu Chen for his guidance.


\end{document}